\documentclass [12pt, letterpaper, oneside] {amsart}

\usepackage {amsmath, amsfonts, amssymb, amscd, setspace, hyperref, color, xy, multirow}
\xyoption{all}

\newcommand \onto {\twoheadrightarrow}

\newcommand \tr {\operatorname {tr}}

\newcommand \Par {\operatorname {Par}}
\newcommand \Tab {\operatorname {Tab}}
\newcommand \Std {\operatorname {Std}}
\newcommand \End {\operatorname {End}}

\newcommand \e {\varepsilon}
\renewcommand \l {\lambda}
\newcommand \s {\sigma}

\newcommand \mc {\mathcal}
\newcommand \mB {\mathbb}
\newcommand \mb {\mathbf}
\newcommand \mt {\mathtt}

\theoremstyle {plain}
\newtheorem {thm} [subsection] {Theorem}

\newtheorem {lemma} [subsection] {Lemma}

\title {A State-sum Formula for the Alexander Polynomial}
\date {\today}

\author {Samson Black}
\address {Department of Mathematics\\ 
University of Oregon\\ 
Eugene, OR 97403}
\email {sblack1@uoregon.edu}


\begin{document}

\begin {abstract}
We develop a diagrammatic formalism for calculating the Alexander polynomial of the closure of a braid as a state-sum.  Our main tools are the Markov trace formulas for the HOMFLY-PT polynomial and Young's semi-normal representations of the Iwahori-Hecke algebras of type A.
\end {abstract}

\maketitle

\section {Introduction}

In \cite {Jones}, Jones gave a construction of the two-variable HOMFLY-PT polynomial invariant using a recursively defined Markov trace on 
certain representations of the braid group.  These representations all factor through the Iwahori-Hecke algebra $\mc{H}_n$ (of type A), which enjoys a character theory \cite {Geck} deforming that of the symmetric group.  The Markov trace can be decomposed as a linear combination of irreducible characters $\chi_\l$ of the Hecke algebra:
\begin {equation}\label{jones}
\tau = \sum_{\l \vdash n} \omega_\l(q,z) \chi_\l.
\end {equation}
The coefficients $\omega_\l(q,z)$ in this ``Fourier expansion'' were calculated by Ocneanu using Schur functions (see, e.g., \cite {GeckJacon}).  By summing only the components corresponding to hook partitions, and specializing $z \to q^{-1}$, the Alexander polynomial $\Delta(L)$ of a link $L$ is recovered.
In this paper, we use (\ref{jones}) to develop a diagrammatic formalism for calculating $\Delta(L)$ as a state-sum derived from a braid presentation of the link $L$.  

The rest of the paper is organized as follows.
In Section \ref {sec s-n}, we review some of the representation theory of the Hecke algebras of type A, using the Hecke algebra analogue of Young's semi-normal form from \cite {Hoefsmit}.  We introduce some combinatorics specific to tableaux of hook shape, and recast the formulas in this language.
In Section \ref {sec constr}, we describe the construction of the state-sum and derive the main theorem (Theorem~\ref {main thm}).  An example is calculated explicitly in Section \ref {sec ex}.  

It is also possible to give a direct proof of our main theorem by verifying that our formulae give a link invariant satisfying the Conway skein relations; see \cite {Black}.  In subsequent work, I hope to generalize these results to colored braids and the multivariable Alexander polynomial.

\emph {Acknowledgements.}  I thank my advisor Arkady Vaintrob for suggesting this work and Jonathan Brundan for showing me that an analogue of Young's semi-normal form exists for Hecke algebras.  I would also like to thank Victor Ostrik for many interesting conversations along the way.

\section {Semi-normal representations of Hecke algebras} \label {sec s-n}

Much of the following is standard (see e.g. \cite {CurtisReiner} or 
\cite [chapters 4 and 5] {KasTur}).  We collect some of the definitions here and fix some notation.

\subsection {Braid group}
Fix $n \geq 1$. Let $\Sigma_n = \{ \s_1, \ldots, \s_{n-1} \} $ be the set of braid generators, and let $\Sigma_n^\bullet$ denote the set of words in the symbols $\Sigma_n \sqcup \Sigma_n^{-1}$.
Let $B_n$ be the braid group on $n$ strands, that is, the quotient of the free group on $\Sigma_n$ by the relations
\begin {eqnarray*}
    \s_r \s_{r+1} \s_r &=& \s_{r+1} \s_r \s_{r+1} \text { for each } 1 \le r \le n-2, \\
    \s_r \s_s &=& \s_s \s_r \text { when } |r - s| > 1.
\end {eqnarray*}
Let $\mc{L}$ be the set of isotopy classes of smooth links in $S^3$.  For each 
$n$, there is a canonical map
$\Sigma_n^\bullet \onto B_n$. Also there is a map
$B_n \rightarrow \mc{L}$ defined by closing a braid into a link.
The map $B_n \to \mB{Z}$, $\s_r \mapsto 1$ induces an isomorphism of groups $B_n / [B_n, B_n] \to \mB{Z}$.  The image of $w \in \Sigma_n^\bullet$ under the composition $\Sigma_n^\bullet \onto B_n \to \mB{Z}$ is called the \emph {exponent sum} of $w$. 

\subsection {Iwahori-Hecke algebra}
Let $\mc{H}_n$ denote the \emph {Iwahori-Hecke algebra} associated to $S_n$.  This is the algebra over $\mB{C}(v)$ generated by $H_1, \ldots, H_{n-1}$, 
subject to the braid relations
\begin {eqnarray*}
    H_r H_{r+1} H_r &=& H_{r+1} H_r H_{r+1} \quad \text { for each } 1 \le r \le n-2, \\
    H_r H_s &=& H_s H_r \quad \text { when } |r - s| > 1,
\end {eqnarray*}
and also the quadratic relations
\begin {equation*}
    (H_r - v)(H_r + v^{-1}) = 0.
\end {equation*}
By setting $q=v^2$ and $T_r = v H_r$ for each $r$, the braid relations look the same in the $T$ variables (they are homogeneous), and the quadratic relations become
\begin {equation*}
    (T_r - q)(T_r + 1) = 0.
\end {equation*}
Because $v$ is generic, it is well known that
$\mc{H}_n$ is a semisimple algebra and its representation theory is equivalent to that of the symmetric group $S_n$ over the field $\mB{C}(v)$.

In what follows, $[r] \in \mB{Z}[v,v^{-1}]$ denotes the quantum integer
\begin {equation}
[r] = \frac {v^r - v^{-r}} {v - v^{-1}} \label {q int}
\end {equation}
for any $r \in \mB{Z}$.

\subsection {Seminormal representations}
This exposition follows \cite[section 3]{Ram}, although the results were originally worked out in \cite {Hoefsmit}.  For a new point of view and substantial generalization, see \cite[section 5]{BK}.

Let $\Par(n)$ denote the set of all integer partitions of $n$.  To the partition $\l = (\l_1 \ge \l_2 \ge \cdots)$, associate its \emph {Young diagram}, that is, the left-justified diagram with $\l_1$ boxes on the first row, $\l_2$ boxes on the second row, etc.

Let $\Tab(\l)$ denote the set of $\l$-tableaux.  These are fillings of the boxes in the Young diagram $\l$ by the numbers $1, \dots, n$.  Let $\Std(\l)$ denote the set of \emph {standard} $\l$-tableaux, namely those that increase across rows and down columns.  The symmetric group $S_n$ acts on $\Tab(\l)$ via its natural action on the entries, although $\Std(\l)$ is not stable under this action.
For a tableau $\mt{T} \in \Tab(\l)$, its {\em residue sequence}
$(i_1,\dots,i_n) \in \mB{Z}^n$ 
is defined by setting $i_r = b-a$ where the box labeled
$r$ in $\mt{T}$ appears in row $a$ and column $b$.

\begin {figure} [h]
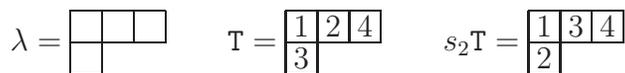

\[
\l = \vcenter {\xy 0;/r.10pc/:
(0,0);(30,0)**@{-},
(0,-10);(30,-10)**@{-},
(0,-20);(10,-20)**@{-},
(0,0);(0,-20)**@{-},
(10,0);(10,-20)**@{-},
(20,0);(20,-10)**@{-},
(30,0);(30,-10)**@{-}
\endxy}
\qquad
\mt{T} = \vcenter {\xy 0;/r.10pc/:
(0,0);(30,0)**@{-},
(0,-10);(30,-10)**@{-},
(0,-20);(10,-20)**@{-},
(0,0);(0,-20)**@{-},
(10,0);(10,-20)**@{-},
(20,0);(20,-10)**@{-},
(30,0);(30,-10)**@{-},
(5,-5)*{1},
(15,-5)*{2},
(5,-15)*{3},
(25,-5)*{4}
\endxy}
\qquad
s_2 \mt{T} = \vcenter {\xy 0;/r.10pc/:
(0,0);(30,0)**@{-},
(0,-10);(30,-10)**@{-},
(0,-20);(10,-20)**@{-},
(0,0);(0,-20)**@{-},
(10,0);(10,-20)**@{-},
(20,0);(20,-10)**@{-},
(30,0);(30,-10)**@{-},
(5,-5)*{1},
(15,-5)*{3},
(5,-15)*{2},
(25,-5)*{4}
\endxy}
\]
\caption {The partition $\l = (3,1) \in \Par(4)$ and standard tableaux $\mt{T}$ and $s_2\mt{T}$.  Here, $\mt{T}$ has residue sequence $(0,1,-1,2)$.}
\end {figure}

Fix a partition $\l \in \Par(n)$ and let $S(\l)$ be the $\mB{C}(v)$-vector space on basis \{$x_{\mt{T}} \;|\; \mt{T} \in \Std(\l)\}$.  Let $(i_1, \ldots, i_n)$ be the residue sequence of $\mt{T}$ and define 
$a_r(\mt{T}), b_r(\mt{T}) \in \mB{C}(v)$ to be
\begin {equation}\label{abs}
a_r(\mt{T}) = \frac {v - v^{-1}} {1 - v^{2(i_r-i_{r+1})}},\qquad
b_r(\mt{T}) = v^{-1}+a_r(\mt{T}).
\end {equation}
Define actions of the generators $H_1, \ldots, H_{n-1}$ of $\mc{H}_n$ on $S(\l)$ by
\begin {equation} \label {s-n formula}
H_r x_{\mt{T}} = a_r (\mt{T})x_{\mt{T}} + b_r (\mt{T}) x_{s_r\mt{T}},
\end {equation}
where we interpret $x_{s_r \mt{T}} = 0$ if $s_r \mt{T}$ is not a standard tableau.

\begin {thm} [Semi-normal representations] \label {irrep thm}
This action extends to make $S(\l)$ into a well-defined $\mc{H}_n$-module.  Furthermore, the modules $\{S(\l) \; | \; \l \in \Par(n) \}$ constitute a complete set of pairwise non-isomorphic irreducible modules for $\mc{H}_n$.
\end {thm}

\subsection {Sign sequences and hook partitions}
For $0 \le \ell \le n-1$, let $\l_\ell$ be the hook partition $(n-\ell, 1^\ell)$.  We refer to $\ell$ as \emph {leg length}.

\begin {lemma} \label {sign seq}
Standard tableaux of shape $\l_\ell$ are in bijection 
with sign sequences $\e =(\e_1,\dots,\e_n)\in \{\pm\}^n$ such that $\e_1 = +$ and
$\ell$ entries equal $-$.
\end {lemma}

\begin {proof}
Beginning with a standard $\l_\ell$-tableau, define $\e = (\e_1, \ldots, \e_n)$ by
\begin {equation}
\e_r = \left\{ \begin {array} {l l}
+, \qquad & \text {if } r \text { appears on the first row} \\
-, \qquad & \text {otherwise}
\end {array} \right.
\end {equation}
Notice that the box labeled 1 has to be in the corner of the hook, so $\e_1 = +$.  Also, $\ell$ numbers are on the leg of the hook, so there are 
$\ell$ entries equal to $-$.

For the inverse, starting with a sign sequence $\e = (\e_1, \ldots, \e_n)$ with $\e_1=+$ and $\ell$ other entries equal to $-$, construct a standard tableau recursively, as follows.  Place 1 in the corner of the diagram.  Now, for each $r>1$, suppose that the numbers $1, \ldots, r-1$ have been placed.  Either add $r$ to the end of the first row or at the bottom of the first column, according to whether $\e_r$ is $+$ or $-$, respectively.
\end {proof}

Using this bijection, we can adapt the semi-normal representation to the combinatorics of sign sequences.

\begin{thm}\label{signs irrep thm}
The irreducible module $S(\l_\ell)$ has basis $\{x_{\e}\}$, where $\e$ 
runs over sign sequences having $\e_1=+$ and $\ell$ other entries equal to $-$.  The generators $H_1, \ldots, H_{n-1}$ of $\mc{H}_n$ act by
\begin {equation} \label {signs s-n formula}
H_r x_\e = a_r(\e) x_\e + b_r(\e) x_{s_r\e}
\end {equation}
where $s_r \e$ denotes the sign sequence obtained from $\e$ by 
permuting $\e_r$ and $\e_{r+1}$, $x_\e$ is interpreted as zero if $\e_1 = -$, and
\begin {align} \label {diag coeff}
a_r(\e) &= \left\{ \begin {array} {l l}
v &\qquad \text {if } (\e_r, \e_{r+1}) = (+, +) \\
-v^{-1} &\qquad \text {if } (\e_r, \e_{r+1}) = (-, -) \\
v^r / [r] &\qquad \text {if } (\e_r, \e_{r+1}) = (-, +) \\
-v^{-r} / [r] &\qquad \text {if } (\e_r, \e_{r+1}) = (+, -),
\end {array} \right.\\
\label {off-diag coeff}
b_r(\e) &= \left\{ \begin {array} {l l}
{[r+1] / [r]} &\qquad \text {if } (\e_r, \e_{r+1}) = (-, +) \\
{[r-1] / [r]} &\qquad \text {if } (\e_r, \e_{r+1}) = (+, -)\\
0&\qquad \text{otherwise.}
\end {array} \right.
\end {align}
The inverse generators $H_1^{-1}, \ldots, H_{n-1}^{-1}$ act by
\begin {equation} \label {newer}
H_r^{-1} x_\e = \bar a_r(\e) x_\e + b_r(\e) x_{s_r\e}
\end {equation}
where $\bar a_r(\e)$ is obtained from 
$a_r(\e)$ by replacing $v$ by $v^{-1}$.
\end{thm}

\begin{proof}
This is just a translation of Theorem~\ref{irrep thm}
using the bijection from Lemma~\ref{sign seq}.
Given a sign sequence
$\e \in \{\pm\}^n$
having $\e_1=+$ and $\ell$ other entries equal to $-$, 
construct the corresponding standard tableau, and let $(i_1,\dots,i_n)$ be its residue sequence.
We have $i_1 = 0$, and for $1 \le r < n$,
\begin {equation} \label {sign to cont}
i_{r+1} = \left\{ \begin {array} {l l}
i_r + 1 &\qquad \text {if } (\e_r, \e_{r+1}) = (+, +) \\
i_r - 1 &\qquad \text {if } (\e_r, \e_{r+1}) = (-, -) \\
i_r + r &\qquad \text {if } (\e_r, \e_{r+1}) = (-, +) \\
i_r - r &\qquad \text {if } (\e_r, \e_{r+1}) = (+, -).
\end {array} \right.
\end {equation}
Given this, the formulae  (\ref{diag coeff})--(\ref{off-diag coeff}) are easily deduced from (\ref{abs}).
Finally the formula (\ref{newer}) is easily deduced from (\ref{signs s-n formula})
since $H_r^{-1} = H_r - (v-v^{-1})$.
\end{proof}

\section {Construction} \label {sec constr}

Begin with a word $w \in \Sigma_n^\bullet$ in the braid generators $\s_1, \ldots, \s_{n-1}$ and their inverses, which we picture as a diagram drawn up the page as the word is read from right to left.  
\[
\s_r = \; \vcenter{
\xy 0;/r.10pc/:
(0,15)*{},
(0,0);(0,10)**@{-},
(10,5)*{\cdots},
(20,0);(20,10)**@{-},
(30,0);(40,10)**@{-},
(40,0);(36,4)**@{-},
(34,6);(30,10)**@{-},
(30,-4)*{_{\phantom{+1} r \phantom {+1}}},
(40,-4)*{_{r+1}},
(50,0);(50,10)**@{-},
(60,5)*{\cdots},
(70,0);(70,10)**@{-}
\endxy
} \qquad \qquad
\s_r^{-1} = \; \vcenter{
\xy 0;/r.10pc/:
(0,15)*{},
(0,0);(0,10)**@{-},
(10,5)*{\cdots},
(20,0);(20,10)**@{-},
(30,0);(34,4)**@{-},
(36,6);(40,10)**@{-},
(40,0);(30,10)**@{-},
(30,-4)*{_{\phantom{+1} r \phantom {+1}}},
(40,-4)*{_{r+1}},
(50,0);(50,10)**@{-},
(60,5)*{\cdots},
(70,0);(70,10)**@{-}
\endxy
}
\]

Construct \emph {permutation diagrams} from the braid diagram by replacing each crossing by one of two resolutions:
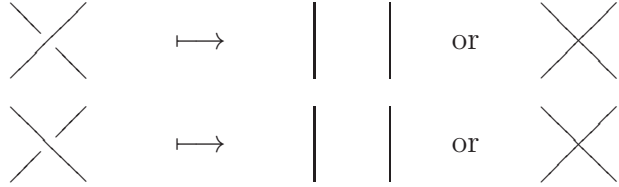
\begin {eqnarray*}
\vcenter{
\xy
(0,0);(10,10)**@{-},
(10,0);(6,4)**@{-},
(4,6);(0,10)**@{-}
\endxy
} \qquad &\longmapsto& \qquad \vcenter{
\xy
(0,0);(0,10)**@{-},
(10,0);(10,10)**@{-}
\endxy
} \qquad \text{or} \qquad \vcenter{
\xy
(0,0);(10,10)**@{-},
(10,0);(0,10)**@{-}
\endxy
} \\ \\
\vcenter{
\xy
(0,0);(4,4)**@{-},
(6,6);(10,10)**@{-},
(10,0);(0,10)**@{-}
\endxy
} \qquad &\longmapsto& \qquad \vcenter{
\xy
(0,0);(0,10)**@{-},
(10,0);(10,10)**@{-}
\endxy
} \qquad \text{or} \qquad \vcenter{
\xy
(0,0);(10,10)**@{-},
(10,0);(0,10)**@{-}
\endxy
}
\end {eqnarray*}
A permutation diagram $\mb{x}$ is \emph {admissible} if
\begin {description}
\item [P1] the first (leftmost) strand goes straight through without crossing any other strands, and
\item [P2] the underlying permutation is the identity.
\end {description}
A \emph{state} is a pair $(\mb{x},\e)$, where $\mb{x}$ is an admissible permutation diagram and $\e$ is an assignment of a sign $\pm$ to each strand 
such that
\begin {description}
\item [S1] the first (leftmost) sign is $+$, and
\item [S2] no two strands of the same sign cross.
\end {description}
Let $\mc{S}(w)$ denote the set of states for $w$.  To a state $(\mb{x},\e) \in \mc{S}(w)$, we associate a weight $M(w, \mb{x}, \e) \in \mB{C}(v)$ defined by multiplying together
certain scalars, one for each resolved crossing.  The scalar associated to a positive crossing of strands in positions $r$ and $r+1$ is given in (\ref{scalars}).  For a negative crossing, replace $v$ by $v^{-1}$ in each expression.

\begin {equation} \label {scalars}
\vcenter{
\xy
(0,0);(10,10)**@{-},
(10,0);(6,4)**@{-},
(4,6);(0,10)**@{-},
(0,-2)*={_r},
(10,-2)*={_{r+1}},
\endxy
} \quad \longmapsto \quad \left \{
\begin {array} {r c r}
v \quad \vcenter{
\xy
(0,-2)*={_+},
(0,0);(0,10)**@{-},
(10,-2)*={_+},
(10,0);(10,10)**@{-}
\endxy
} & \qquad &
-v^{-1} \quad \vcenter{
\xy
(0,-2)*={_-},
(0,0);(0,10)**@{-},
(10,-2)*={_-},
(10,0);(10,10)**@{-}
\endxy
} \\ & & \\ & & \\
\dfrac {v^r} {[r]} \quad \vcenter{
\xy
(0,-2)*={_-},
(0,0);(0,10)**@{-},
(10,-2)*={_+},
(10,0);(10,10)**@{-}
\endxy
} & \qquad &
\dfrac {-v^{-r}} {[r]} \quad \vcenter{
\xy
(0,-2)*={_+},
(0,0);(0,10)**@{-},
(10,-2)*={_-},
(10,0);(10,10)**@{-}
\endxy
} \\ & & \\ & & \\
\dfrac {[r+1]} {[r]} \quad \vcenter{
\xy
(0,-2)*={_-},
(0,0);(10,10)**@{-},
(10,-2)*={_+},
(10,0);(0,10)**@{-}
\endxy
} & &
\dfrac {[r-1]} {[r]} \quad \vcenter{
\xy
(0,-2)*={_+},
(0,0);(10,10)**@{-},
(10,-2)*={_-},
(10,0);(0,10)**@{-}
\endxy
}
\end {array}
\right.
\end {equation} \\
Define $A(w) \in \mB{C}(v)$ by
\begin {equation} \label {state-sum}
A(w)=\frac 1 {[n]} \sum_{(\mb{x}, \e) \in \mc{S}(w)} \langle \e \rangle M(w,\mb{x},\e)
\end {equation}
where $\langle \e \rangle \in \{ \pm 1 \}$ is the product of the signs $\e_1, \ldots, \e_n$ attached to the strands.
\begin {thm} \label {main thm}
Let $L$ be an oriented link, and let $w \in \Sigma_n^\bullet$ 
represent a braid in $B_n$ whose closure is $L$.  Then, $A(w)$ is a polynomial in $v-v^{-1}$, and
\[ A(w) = \Delta(L), \] where $\Delta(L)$ 
is the Conway-normalized Alexander polynomial.
\end {thm}

\begin {proof}
To avoid confusion when switching between the generators
$H_r$ and $T_r = vH_r$ of $\mc{H}_n$, let us write
$\varphi: B_n \onto \mc{H}_n^\times$ for the group homomorphism 
given by $\varphi(\s_r) = H_r$ and $\psi: B_n \onto \mc{H}_n^\times$
for the one with $\psi(\s_r) = T_r$.
Formula (7.2) in \cite {Jones} gives the Alexander polynomial for a link $L$ as
\begin{equation}
    \Delta(L) = (-1)^{n-1} \left( \frac 1 q \right)^{(e-n+1)/2} \frac {1-q} {1-q^n} \sum_{k=0}^{n-1} (-1)^k \chi_{n-1-k} \bigl( \psi(w) \bigr)
\end{equation}
where $\chi_\ell = \tr \circ  \rho_{\ell}$ 
is the character of $\mc{H}_n$ arising from the irreducible representation $\rho_\ell: \mc{H}_n \to \End(S(\l_\ell))$ indexed by the hook partition $(n-\ell, 1^\ell)$, and $e$ is the exponent sum of $w$.
Put $q = v^2$ and $T_r=vH_r$ for each $r$, so that $\psi(w) = v^e \varphi(w)$.  Reindex the sum over $\ell = n-1-k$ to get
\begin{eqnarray}
    \Delta(L) &=& (-1)^{n-1} \left( \frac 1 v \right)^{e-n+1} \frac {1-v^2} {1-v^{2n}} \sum_{\ell=0}^{n-1} (-1)^{n-1-\ell} \chi_\ell \bigl( v^e \varphi(w) \bigr) \nonumber \\
    &=&  \frac 1 {[n]} \sum_{\ell=0}^{n-1} (-1)^\ell \chi_\ell \bigl( \varphi(w) \bigr).
\end{eqnarray}

Now we compute $\chi_\ell$ by using the semi-normal form for
$S(\l_\ell)$.
The action of the generators $H_r$, $H_r^{-1}$ on $x_{\e}$ from Theorem~\ref {signs irrep thm} are pictured in (\ref{s-n signs}).
\begin {equation} \label {s-n signs}
\begin {array} {c c c}
\vcenter{
\xy
(0,0);(10,10)**@{-},
(10,0);(6,4)**@{-},
(4,6);(0,10)**@{-},
(0,-3)*{_r},
(10,-3)*{_{r+1}},
(0,13)*{}
\endxy
} \quad &\longmapsto& \quad a_r(\e) \vcenter{
\xy
(0,0);(0,10)**@{-},
(10,0);(10,10)**@{-},
(0,-3)*{\e_r},
(0,13)*{\e_r},
(10,-3)*{\e_{r+1}},
(10,13)*{\e_{r+1}}
\endxy
} \quad + \quad b_r(\e) \vcenter{
\xy
(0,0);(10,10)**@{-},
(10,0);(0,10)**@{-},
(0,-3)*{\e_r},
(10,13)*{\e_r},
(10,-3)*{\e_{r+1}},
(0,13)*{\e_{r+1}}
\endxy
} \\
&& \\
\vcenter{
\xy
(0,0);(4,4)**@{-},
(6,6);(10,10)**@{-},
(10,0);(0,10)**@{-},
(0,-3)*{_r},
(10,-3)*{_{r+1}},
(0,13)*{}
\endxy
} \quad &\longmapsto& \quad \bar a_r(\e) \vcenter{
\xy
(0,0);(0,10)**@{-},
(10,0);(10,10)**@{-},
(0,-3)*{\e_r},
(0,13)*{\e_r},
(10,-3)*{\e_{r+1}},
(10,13)*{\e_{r+1}}
\endxy
} \quad + \quad b_r(\e) \vcenter{
\xy
(0,0);(10,10)**@{-},
(10,0);(0,10)**@{-},
(0,-3)*{\e_r},
(10,13)*{\e_r},
(10,-3)*{\e_{r+1}},
(0,13)*{\e_{r+1}}
\endxy
}
\end {array}
\end {equation}
Moreover, if $r = 1$, the second term on the right hand side should be omitted.
Only diagonal entries of the matrix $\rho_\ell(\varphi(w))$ contribute to the trace.  Hence, for each $\ell$ we need only consider those permutation diagrams that represent the identity permutation and whose first strand goes straight through.  The theorem follows on comparing formulas (\ref {diag coeff}) and (\ref {off-diag coeff}) with (\ref {scalars}).
\end {proof}

\section {Example} \label {sec ex}

Let's use the braid presentation $w = \s_2^{-1} \s_1 \s_2^{-1} \s_1$ for the figure-eight knot, pictured below with its six possible states.
\[
\vcenter{\xy 0;/r.15pc/:
(0,0);(14,14)**@{-},
(16,16);(20,20)**@{-},
(20,20);(20,30)**@{-},
(20,30);(10,40)**@{-},
(10,0);(6,4)**@{-},
(4,6);(0,10)**@{-},
(0,10);(0,20)**@{-},
(0,20);(14,34)**@{-},
(16,36);(20,40)**@{-},
(20,0);(20,10)**@{-},
(20,10);(6,24)**@{-},
(4,26);(0,30)**@{-},
(0,30);(0,40)**@{-}
\endxy} \qquad
\begin {array} {c | c | c c c}
\mb{x} & \e & M(w, \mb{x}, \e) & & \\
\hline
& & & & \\
\multirow {7} {*}
{\quad $\vcenter{\xy 0;/r.10pc/:
(0,0);(0,40)**@{-},
(10,0);(10,40)**@{-},
(20,0);(20,40)**@{-}
\endxy}$} \quad
& + \, + \, + & v^{-1} \cdot v \cdot v^{-1} \cdot v &=& 1 \\
&&&& \\
&+ \, + \, - & \frac {-v^2} {[2]} \cdot v \cdot \frac {-v^2} {[2]} \cdot v &=& \frac {v^6} {[2]^2} \\
&&&& \\
&+ \, - \, + & \frac {v^{-2}} {[2]} \cdot (-v^{-1}) \cdot \frac {v^{-2}} {[2]} \cdot (-v^{-1}) &=& \frac {v^{-6}} {[2]^2} \\
&&&& \\
&+ \, - \, - & (-v) \cdot (-v^{-1}) \cdot (-v) \cdot (-v^{-1}) &=& 1 \\
&&&& \\
\hline
& & & & \\
\multirow {3} {*}
{\quad $\vcenter{\xy 0;/r.10pc/:
(0,0);(0,40)**@{-},
(10,0);(10,10)**@{-},
(10,10);(20,20)**@{-},
(20,20);(20,30)**@{-},
(20,30);(10,40)**@{-},
(20,0);(20,10)**@{-},
(20,10);(10,20)**@{-},
(10,20);(10,30)**@{-},
(10,30);(20,40)**@{-}
\endxy}$} \quad
&+ \, + \, - & \frac {[3]} {[2]} \cdot (-v^{-1}) \cdot \frac {[1]} {[2]} \cdot v &=& \frac {-[3]} {[2]^2} \\
&&&& \\
&+ \, - \, + & \frac {[1]} {[2]} \cdot v \cdot \frac {[3]} {[2]} \cdot (-v^{-1}) &=& \frac {-[3]} {[2]^2}
\end {array}
\] \\

Now, we calculate the sum, minding the signs associated to each state and the global rescaling.
\begin {equation*}
A(w) = \dfrac {1} {[3]} \left( 2 - \dfrac {v^6 + v^{-6}} {[2]^2} + \dfrac {2 [3]} {[2]^2} \right) = -v^2+3-v^{-2} = 1 - (v-v^{-1})^2
\end {equation*}

\bibliographystyle {plain}
\bibliography {statesumnew}

\end {document}